\documentclass[a4paper,12pt]{amsart}
\usepackage{amsfonts}
\usepackage{amsmath}
\usepackage{amsthm}
\usepackage{amssymb}
\usepackage{mathrsfs}
\usepackage{latexsym}
\usepackage{graphicx}
\usepackage{epic}
\usepackage{float}
\usepackage{vmargin}
\setmargrb{3cm}{2.5cm}{2.5cm}{3cm} 
\begin{document}


\newtheorem{theorem}{Theorem}[section]
\newtheorem{definition}[theorem]{Definition}
\newtheorem{lemma}[theorem]{Lemma}
\newtheorem{proposition}[theorem]{Proposition}
\newtheorem{corollary}[theorem]{Corollary}
\newtheorem{example}[theorem]{Example}
\newtheorem{remark}[theorem]{Remark}
\newtheorem{acknowledgement}[theorem]{Acknowledgement}

\newcommand{\cAtheta}{\mathcal{A}_\theta}
\newcommand{\cAab}{\mathcal{A}_{\alpha\beta}}

\newcommand{\cA}{\mathcal{A}}

\newcommand{\mcA}{\mathcal{A}}

\newcommand{\Alafgra}{{_\mathcal{A}}\langle f,g\rangle}
\newcommand{\lafgraB}{\langle f,g\rangle_\mathcal{B}}

\newcommand{\AiLa}{\mathcal{A}^1(\Lambda,c)}
\newcommand{\AisLa}{\mathcal{A}^1_s(\Lambda,c)}
\newcommand{\Aiab}{\mathcal{A}^1(\alpha\ZZ\times\beta\ZZ,c)}
\newcommand{\Aiba}{\mathcal{A}^\infty(\beta^{-1}\ZZ\times\alpha^{-1},c)}
\newcommand{\AiLacirc}{\mathcal{A}^1(\Lambda^\circ,\overline{c})}
\newcommand{\AisLacirc}{\mathcal{A}^1_s(\Lambda^\circ,\overline{c})}
\newcommand{\AinfLa}{\mathcal{A}^\infty(\Lambda,c)}
\newcommand{\AinfsLa}{\mathcal{A}^\infty_s(\Lambda,c)}
\newcommand{\Ainfab}{\mathcal{A}^\infty(\alpha\ZZ\times\beta\ZZ,c)}
\newcommand{\AinfLacirc}{\mathcal{A}^\infty(\Lambda^\circ,\overline{c})}
\newcommand{\AinfsLacirc}{\mathcal{A}^\infty_s(\Lambda^\circ,\overline{c})}
\newcommand{\Ainfba}{\mathcal{A}^\infty(\beta^{-1}\ZZ\times\alpha^{-1},c)}
\newcommand{\Ata}{\mathcal{A}_\theta}
\newcommand{\aZZ}{\alpha\mathbb{Z}}
\newcommand{\bZZ}{\beta\mathbb{Z}}

\newcommand{\akl}{a_{kl}}
\newcommand{\amn}{a_{mn}}
\newcommand{\bkl}{b_{kl}}
\newcommand{\bmn}{b_{mn}}
\newcommand{\bmnkl}{b_{m-k,n-l}}
\newcommand{\bfa}{{\bf{a}}}

\newcommand{\bfb}{{\bf{b}}}
\newcommand{\cB}{\mathcal{B}}
\newcommand{\mcB}{\mathcal{B}}

\newcommand{\mcC}{\mathcal{C}}

\newcommand{\CmS}{\mathrm{C}^*}
\newcommand{\CstLatw}{\mathrm{C}^*(\Lambda,c)}
\newcommand{\CstLacirctw}{\mathrm{C}^*(\Lambda^\circ,\overline{c})}
\newcommand{\cAVcB}{{_\mathcal{A}}{V}_\mathcal{B}}

\newcommand{\bff}{\mathbf{f}}

\newcommand{\bg}{\mathbf{g}}
\newcommand{\gtl}{g(t;l)}
\newcommand{\pq}{\tfrac{p}{q}}
\newcommand{\trq}{\tfrac{r}{q}}

\newcommand{\GabFrame}{A\|f\|_2^2\le\sum_{\lambda\in\Lambda}|\langle f,\pi(\lambda)g\rangle|^2\le B\|f\|_2^2}
\newcommand{\GgLa}{\mathcal{G}(g,\Lambda)}
\newcommand{\Ggab}{\mathcal{G}(g,\alpha\ZZ\times\beta\ZZ)}
\newcommand{\Ggba}{\mathcal{G}(g,\beta^{-1}\ZZ\times\alpha^{-1}\ZZ)}
\newcommand{\Ghab}{\mathcal{G}(h,\alpha\ZZ\times\beta\ZZ)}
\newcommand{\Ghba}{\mathcal{G}(h,\beta^{-1}\ZZ\times\alpha^{-1}\ZZ)}
\newcommand{\Ggrq}{\mathcal{G}_{r,q}(g)}

\newcommand{\cH}{\mathcal{H}}
\newcommand{\cHq}{\mathcal{H}_q}
\newcommand{\Cs}{C^*}
\newcommand{\alan}{{_\mathcal{A}\langle}}
\newcommand{\lan}{\langle}
\newcommand{\ranb}{\rangle_\mathcal{B}}
\newcommand{\ran}{\rangle}

\newcommand{\la}{\lambda}
\newcommand{\lac}{\lambda^\circ}
\newcommand{\La}{\Lambda}
\newcommand{\Lac}{\Lambda^\circ}

\newcommand{\livsp}{\ell^1_v}
\newcommand{\LtR}{L^2(\mathbb{R})}
\newcommand{\LinfR}{L^\infty(\mathbb{R})}

\newcommand{\Mbq}{M^\beta_q}
\newcommand{\Tarq}{T^\alpha_{r,q}}
\newcommand{\pirq}{\pi_{r,q}}
\newcommand{\pizo}{\pi_{0,1}}

\newcommand{\piab}{\pi(\alpha k,\beta l)}
\newcommand{\piabg}{\pi(\alpha k,\beta l)g}

\newcommand{\RRt}{\newcommand{R}^2}
\newcommand{\RRd}{\mathbf{R}^d}
\newcommand{\Sgab}{S_{g}^{\alpha,\beta}}
\newcommand{\cS}{\mathscr{S}}
\newcommand{\cSR}{\mathscr{S}(\RR)}

\newcommand{\tr}{\mathrm{tr}}

\newcommand{\aV}{{_\mathcal{A}}V}
\newcommand{\Vb}{V_{\mathcal{B}}}
\newcommand{\aVb}{{_\mathcal{A}}V_{\mathcal{B}}}
\newcommand{\Vzo}{V_{0,1}}
\newcommand{\Vpq}{V_{p,q}}
\newcommand{\WLinfli}{W(L^\infty,\ell^1)}
\newcommand\WLinfliv{{ W(L^\infty,\ell^1_v) }}   
\newcommand{\R}{\mathbf{R}}

\newcommand{\CC}{\mathbb{C}}
\newcommand{\CCq}{\mathbb{C}^q}
\newcommand{\RR}{\mathbb{R}}
\newcommand{\Rt}{\mathbb{R}^2}
\newcommand{\RRtd}{\mathbb{R}^{2d}}
\newcommand{\TT}{\mathbb{T}}
\newcommand{\ZZ}{\mathbb{Z}} 
\newcommand{\ZZt}{\mathbb{Z}^2}
\newcommand{\aZbZ}{\alpha\mathbb{Z}\times\beta\mathbb{Z}}
\newcommand{\bZaZ}{\beta^{-1}\ZZ\times\alpha^{-1}\ZZ}

\newcommand{\pa}{\partial}
\newcommand{\opa}{\overline{\partial}}
\newcommand{\onabla}{\overline{\nabla}}
\newcommand{\cE}{\mathcal{E}}
\newcommand{\mi}{\mathrm{i}}
\newcommand{\thZZ}{\theta\mathbb{Z}\times\mathbb{Z}}
\newcommand{\ZthZ}{\mathbb{Z}\times\theta^{-1}\mathbb{Z}}
\newcommand{\pithkl}{\pi(\theta k,l)}
\newcommand{\pikthl}{\pi(k,\theta^{-1}l)}
\newcommand{\SR}{\mathscr{S}(\R)}

\newcommand{\beq}{\begin{equation}}
\newcommand{\eeq}{\end{equation}}
\newcommand{\nn}{\nonumber}
\def\hs#1#2{\left\langle #1,#2\right\rangle}  %
\def\lhs#1#2{{_\bullet\!\!}\left\langle #1,#2\right\rangle}
\def\rhs#1#2{\left\langle #1,#2\right\rangle\!\!{_\bullet}}
\pagestyle{plain}
\title{The Balian-Low theorem and noncommutative tori}
\author{Franz Luef} 
\address{Department of Mathematical Sciences\\ NTNU Trondheim\\7041 Trondheim\\Norway}
\email{franz.luef@math.ntnu.no}

\begin{abstract}
We point out a link between the theorem of Balian and Low on the non-existence of well-localized Gabor-Riesz bases and a constant curvature connection on projective modules over noncommutative tori. 
\end{abstract}
\maketitle \pagestyle{myheadings} \markboth{F. Luef}{The Balian-Low theorem and noncommutative tori}
\thispagestyle{empty}
\section{Introduction}

The theorem of Balian-Low on the non-existence of well-localized Gabor-Riesz bases for $\LtR$ is one of the cornerstones of time-frequency analysis \cite{ba81-1,lo85}. 
For the formulation we first introduce the multiplication operator and differentiation operator, denoted by $(\nabla_1 g)(t)=2\pi \mi\, t\,g(t)$ and $(\nabla_2 g)(t)=g^\prime(t)$ , respectively. Let $\pi(z)g(t)=e^{2\pi i\omega t}g(t-x)$ be the time-frequency shift of a function $g$ by $z=(x,\omega)$ in phase space. Gabor 
studied in \cite{ga46} systems of the form $\mathcal{G}(g,\thZZ)=\{\pithkl g:\,k,l\in\ZZ\}$, so-called {\it Gabor systems} with {\it Gabor atom} $g$. The density theorem for Gabor frames says that if $\mathcal{G}(g,\thZZ)$ $\theta$ is a frame, then $\theta\in(0,1]$. 
\vskip.3cm
\noindent
A natural question about Gabor systems $\mathcal{G}(g,\theta\ZZ\times\ZZ)$ for a fixed Gabor atom $g$ is to study for which $\theta$ the system $\mathcal{G}(g,\theta\ZZ\times\ZZ)$ is a frame. If $g$ is well-localized in time and frequency, then $\theta=1$ will be excluded as for the Gaussian. It turns out that for some Gabor atoms $g$, such as the Gaussian or any totally positive function of finite type $\mathcal{G}(g,\theta\ZZ\times\ZZ)$ is a Gabor frame for any $\theta$ in $(0,1)$, \cite{grst13,ly92,se92-1}. On the other hand the answer for the indicator function of an interval $[0,c]$ is much more intricate and the values $\theta$ and $c$ for which one gets a Gabor frame are known as Janssen's tie \cite{dasu13,ja03}. The theorem of Balian-Low provides an explanation for these facts.
\begin{theorem}[Balian-Low]\label{balianlow}
Suppose the Gabor system $\mathcal{G}(g,\ZZt)$ is an orthonormal basis for $\LtR$. Then 
\begin{equation*}
  \Big(\int_{\RR}|(\nabla _1 g)(t)|^2\,dt\Big)\Big(\int_{\RR}|(\nabla _2 g)(t)|^2\,dt\Big)=\infty. 
\end{equation*}
\end{theorem}
\noindent 
In particular, the theorem of Balian-Low implies that if $\mathcal{G}(g,\ZZ^2)$ is an orthonormal basis for $\LtR$, then $g$ is not well-localized in time and frequency, e.g. $g$ cannot be in the Schwartz class $\cSR$ or in Feichtinger's algebra $S_0(\RR)$. The definition of Gabor systems does not indicate a link to regularity properties of the Gabor atom. Hence, the very reason for the incompatibility between orthonormal Gabor bases of the form $\mathcal{G}(g,\ZZt)$ and good time-frequency localization is not well understood despite the vast literature on the Balian-Low theorem \cite{asfeka14,ba81-1,ba88,behewa95,ga08,grhaheku02,grorro15,niol13,niol17}. Note that  some authors refer to statements of the form Theorem \eqref{balianlow} as {\it weak} Balian-Low theorems.  
\vskip.3cm
\noindent
The main aim of this investigation is to present an approach to Gabor frames that provides an explanation of the link between regularity properties of Gabor atoms and their behavior at the critical density. We are building on the correspondence between Gabor frames and projective modules over noncommutative tori \cite{lu09,lu11}. The standard argument to demonstrate Theorem \ref{balianlow} is due to Battle \cite{ba88}. We are demonstrating that Battle's argument is best understood in terms of noncommutative geometry.
\section{Noncommutative tori}

In noncommutative geometry one attempts to define geometric objects and notions for general $C^*$-algebras. For our purpose we need the noncommutative torus $\cA_\theta$ equipped with 
its structure as a noncommutative manifold. We briefly recall the construction of vector bundles over noncommutative tori, which are finitely generated projective 
modules over $\cA_\theta$, the differential structure on $\cA_\theta$ is given by derivations and on the vector bundles by a connection, and the notion of curvature of a connection \cite{co80}.   
 \vskip.3cm
 \noindent
We denote the operators $\pi(0,1)$ and $\pi(\theta,0)$  as $M_1$ and $T_\theta$,
respectively. Note that we have $M_1T_\theta=e^{2\pi i\theta}T_\theta M_1$ and hence the norm closure of $\{\pi(k\theta,l):\,k,l\in\ZZ\}$ 
defines the noncommutative torus $\cA_\theta$, \cite{ri81}. The smooth noncommutative torus is the subalgebra $\cA_\theta^\infty$ of $\cA_\theta$ consisting of operators
\begin{equation}\label{smoothA}
  \pi({\bf a})=\sum_{k,l\in\ZZ}a_{kl}\pithkl, \qquad \textup{for} \quad {\bf a}=(a_{kl})\in\mathscr{S}(\ZZt).
\end{equation}
The standard derivations on $\cA_\theta$ are given by
\begin{align}\label{2dernct}
\pa_1(a)  = 2\pi \mi \, 
\sum_{k,l}k a_{kl}\pithkl \, \nonumber \quad \textup{and} \qquad  
\pa_2(a) =2\pi \mi \, 
\sum_{k,l}l a_{kl}\pithkl \, . 
\end{align}
The Schwartz space $\cSR$ turns out to be vector bundle over $\cA_\theta^\infty$, \cite{co80,lu09,ri88}.
\begin{proposition}[Connes]
The Schwartz space $\cSR$ is a finitely generated projective module over the smooth noncommutative torus $\cA_\theta^\infty$ with respect to the following $\cA_\theta^\infty$-valued inner product and left action:
\begin{align}
	 \lhs{f}{g}    &=\sum_{k,l\in\ZZ}\langle f,\pithkl g\rangle\pithkl  \quad\text{for}\, f,g\in\cS(\RR),\nonumber \\
		{\bf a}\cdot g&=\sum_{k,l\in\ZZ}a_{kl}\pithkl g\quad\text{for}\, a\in\cS(\ZZt),\,g\in\cSR .\nonumber 
\end{align}
\end{proposition}
\vskip.3cm
\noindent
Connes defined also a constant curvature connection on $\cSR$,  see \cite{co80}, given by covariant derivatives 
$\nabla_1$, $\nabla_2$:
\beq\label{ccctorus}
  (\nabla_1 g)(t)=2\pi \mi \, \theta^{-1} \, t \, g(t) \qquad \textup{and} \qquad (\nabla_2 g)(t)=g^\prime(t) \, .
\eeq
The covariant derivatives satisfy the left Leibniz rule
\beq
\nabla_i (a\cdot g)=(\partial_i a)\cdot g+a\cdot(\nabla_i g), \qquad i =1,2 .
\eeq
For example, we have for $a=\pithkl$ that $\nabla_1(\pithkl g)=2\pi i k\pithkl g+\pithkl\nabla_1g$.
\vskip.3cm
\noindent
The covariant derivatives are compatible with the hermitian structure of the $\cA_\theta^\infty$-module $\cSR$:
\begin{equation*}
\partial_i (\lhs{f}{g})=\lhs{\nabla_i f}{g} + \lhs{f}{\nabla_i g}, \qquad i=1,2 .
\end{equation*}
Finally we observe that the connection has constant curvature:
\beq\label{cocurv}
F_{1,2} := [\nabla_1, \nabla_2] = - 2 \pi\mi \, \theta^{-1} \, \mathrm{I}_{\cSR} \, , 
\eeq
which acts on the left on the $\cA_\theta^\infty$-module $\cSR$, see \cite{co80,dalalu15}. 
In the case of the Moyal plane we have a similar setting and the constant curvature connection provides a gauge-theoretic 
description of the canonical commutation relations \cite{dalalu15}, which also sheds some additional light on Battle's proof 
in the next section.
\section{Balian-Low Theorem}

The theorem of Connes that $\cSR$ is a finitely generated projective module over $\cA_\theta$ implies that there exist generators $g_1,...,g_n$ in $\cSR$ 
such that 
\begin{equation*}
  f=\lhs{f}{g_1}g_1+\cdots+\lhs{f}{g_n}g_n~~\text{for all}~~f\in\cSR.
\end{equation*}
In this investigation we restrict our interest to the case, when the $\cA_\theta^\infty$-module $\cSR$ has one generator $g$. 
In \cite{lu09} it was shown that this is equivalent to the Gabor system $\mathcal{G}(g,\theta\ZZ\times\ZZ)$ to be a frame for $\LtR$, i.e. there exist 
constants $A$ and $B$, so called frame constants, such that 
\begin{equation*}
   A\|f\|_2^2\le\sum_{k,l}|\langle f,\pithkl g\rangle|^2\le B\|f\|_2^2~~\text{for all}~~f\in\LtR,
\end{equation*}
where $A$ is the maximal and $B$ the minimal constant with these properties. Equivalently, we have that the frame operator $Sf:=\lhs{f}{g}{g}$ has spectrum contained in $[A,B]$. 
\vskip.3cm
\noindent
If $\mathcal{G}(g,\theta\ZZ\times\ZZ)$ is a Gabor frame, then 
one has expansions of the form
\begin{equation*}
   f=\sum_{k,l}\langle f,\pithkl g\rangle\pithkl S^{-1}g=\sum_{k,l}\langle f,\pithkl S^{-1}g\rangle\pithkl g
\end{equation*}
for all $f\in\LtR$. The atom $S^{-1}g$ in the above expansions is known as the canonical dual Gabor atom. We showed in \cite{lu11} that there is 
a correspondence between tight Gabor frames, i.e. Gabor frames $\mathcal{G}(g,\theta\ZZ\times\ZZ)$ with $A=B$ and atoms $g$ in $\cSR$, and projections $p=\lhs{g}{g}$ in noncommutative tori $\cA_\theta^\infty$.  
\vskip.3cm
\noindent
 The Balian-Low theorem is a statement on the finer properties of the generators for the $\cA_\theta^\infty$-module $\cSR$ in the case $\theta=1$. Note that $\cA_1$ is the commutative $C^*$-algebra $C(\TT^2)$ of continuous functions over the torus $\TT^2$ and that the smooth subalgebra $\cA_1^\infty$ is the space of smooth functions over the torus $C^\infty(\TT^2)$. 
\vskip.3cm
\noindent
Our main contribution is to place Battle's proof of the theorem of the (weak) Balian-Low Theorem into the framework of noncommutative geometry. Namely, Battle 
uses the left Leibniz rule for the covariant derivations $\nabla_1$ and $\nabla_2$ for the time-frequency shifts $\pi(k,l)$ and that the 
connection defined by $\nabla_1,\nabla_2$ has constant curvature $-2\pi\mi$.
\vskip.3cm
\noindent
We present now Battle's argument for a more general formulation of the theorem of Balian-Low \cite{grhaheku02}.
\begin{theorem}[Weak-Balian-Low]
Suppose $\mathcal{G}(g,\theta\ZZ\times\ZZ)$ for a $g\in\LtR$ is a Riesz basis for $\LtR$ and 
we denote by $h$ the canonical dual Gabor atom $S^{-1}g$. Then either $\Big(\int_{\RR}|(\nabla _1 g)(t)|^2\,dt\Big)\Big(\int_{\RR}|(\nabla _2 g)(t)|^2\,dt\Big)=\infty$ 
or $\Big(\int_{\RR}|(\nabla _1 h)(t)|^2\,dt\Big)\Big(\int_{\RR}|(\nabla _2 h)(t)|^2\,dt\Big)=\infty$. 
\end{theorem}
\begin{proof}
  Suppose $\nabla_i g$ and $\nabla_i h$ are in $\LtR$ for $i=1,2$. Then the left Leibniz rule for $\nabla_1$ applied to $\pi(k,l)$ implies that
	\begin{equation*}
	  \nabla_1(\pi(k,l)g)=2\pi i k\pi(k,l)g+\pi(k,l)\nabla_1g.
	\end{equation*}
	Now, the biorthogonality of $g$ and $h$ yields that the second term on the right hand side vanishes 
	\begin{equation*}
	  \langle\nabla_1 g,\pi(k,l)h\rangle=\langle g,\nabla_1\pi(k,l)h\rangle=2\pi i k\langle g,\pi(k,l)h\rangle+\langle g,\pi(k,l)\nabla_1h\rangle
	\end{equation*}
	and consequently
	\begin{equation*}
	  \langle\nabla_1 g,\pi(k,l)h\rangle=\langle \pi(-k,-l)g,\nabla_1 h\rangle
	\end{equation*}
and in a similar manner one shows that $\langle\pi(k,l)g,\nabla_2 h\rangle=\langle \nabla_2 g,\pi(-k,-l)h\rangle$.
	By assumption $\mathcal{G}(g,\theta\ZZ\times\ZZ)$ is a Riesz basis for $\LtR$ and $h=S^{-1}g$ is the unique dual Gabor atom. Hence we have 
	\begin{equation*}
	  \nabla_1 g=\sum_{k,l}\langle\nabla_1g,\pi(k,l)h\rangle\pi(k,l)g\quad\textup{and}\quad\nabla_2 h=\sum_{k,l}\langle\nabla_2 h,\pi(k,l)g\rangle\pi(k,l)h.
	\end{equation*}
	We use these relations to derive at a contradiction to the fact that $\nabla_1,\nabla_2$ is a constant curvature connection on $\cSR$.
	\begin{align}
	  \langle \nabla_1 g,\nabla_2 h\rangle&=\Big\langle \sum_{k,l}\langle\nabla_1g,\pi(k,l)h\rangle\pi(k,l)g,\nabla_2h\Big\rangle\nonumber\\
		                                    &=\sum_{k,l}\langle\pi(-k,-l)g,\nabla_1 h\rangle\langle \nabla_2 g,\pi(-k,-l)h\rangle\nonumber\\
																				&=\sum_{m,n}\langle\pi(m,n)g,\nabla_1 h\rangle\langle \nabla_2 g,\pi(m,n)h\rangle\nonumber\\
																				&=\sum_{m,n}\langle \nabla_2 g,\pi(m,n)h\rangle\langle\pi(m,n)g,\nabla_1 h\rangle\nonumber\\
																				&=\Big\langle\nabla_2g,\sum_{m,n}\langle\nabla_1 h,\pi(m,n)h\rangle\pi(m,n)g \Big\rangle\nonumber\\
																				&=\langle \nabla_2 g,\nabla_1 h\rangle
	\end{align}
	yields that 
	\[\langle (\nabla_1\nabla_2-\nabla_2\nabla_1) g,h\rangle=-2\pi i\theta^{-1}\langle g,h\rangle.\]
Hence the curvature $F_{1,2}=\nabla_1\nabla_2-\nabla_2\nabla_1$ has a kernel as an operator on $L^2(\mathbb{R})$ that vanishes which is a contradiction to $F_{1,2}=-2\pi i\theta^{-1}I$. Hence we  have arrived at a contradiction to the assumption $\nabla_i g$ and $\nabla_i h$ are in $\LtR$ for $i=1,2$.
\end{proof}
In particular, the Balian-Low theorem implies that a $g\in\cSR$ cannot generate a Riesz basis of the form $\mathcal{G}(g,\ZZ^2)$. For $\theta\in(0,1)$ there exist Schwartz functions, e.g. the Gaussian $g(t)=e^{-\pi t^2}$, that generate Gabor frames for $\theta\ZZ\times\ZZ$. These Gabor frames may be used to construct non-trivial projections in $\cA_\theta$ with trace $\theta$, see \cite{lu11}. In contrast to the commutative case, where $\cA_{\theta=1}$ only has the trivial projections $p=0$ and $p=I$. Recall that $\cA_1^\infty$ is Morita equivalent to $C^\infty(\TT^2)$. Hence one may conclude that the generators of line bundles over noncommutative tori $\cA^\infty_\theta$ for $\theta\in(0,1)$ are more well-behaved, then the line bundles over $C^{\infty}(\TT^2)$. 
\begin{theorem}[Weak-Balian-Low-Revisited]
Let $g$ be a generator of the $\cA_1^\infty=C^\infty(\TT^2)$-module $\cSR$. Suppose the generator yields a Riesz basis for $\cSR$ and $h=S^{-1}g$ its canonical dual atom. 
Then either $\nabla_1g$ or $\nabla_2 h$ are not $\cSR$.
\end{theorem}
Finally, there is an extension of the Balian-Low theorem to the case of Gabor-Riesz bases generated by atoms in Feichtinger's algebra $S_0(\RR)$, see \cite{asfeka14,grma13}. Consequently, the preceding statement also holds for singly-generated projective $S_0(\RR)$ modules over the subalgebra $\cA^1_\theta= \{{\bf a}=\sum_{k,l\in\ZZ}a_{kl}\pithkl: a\in\ell^1(\ZZt)\}$ of $\cA_\theta$. Note that $\cA_1^1$ is the Wiener algebra $\cA(\TT^2)$ of absolutely convergent Fourier series. 
\begin{theorem}
Let $g$ be a generator of the $\cA_1^1=\cA(\TT^2)$-module $S_0(\RR)$. Then $g$ is not in $S_0(\RR)$.
\end{theorem}
There is an argument based on results from operator algebras. We know from \cite{lu09,lu11} that generators of $\cA_\theta^1$-modules are given by Gabor frames with generators in $S_0(\RR)$. 
\begin{proof} 
  Suppose $g\in S_0(\RR)$ generates a Gabor frame $\mathcal{G}(g,\ZZ^2)$. Then its canonical dual atom $h$ is also in $S_0(\RR)$ and 
	\[p=\sum_{k,l\in\ZZ^2}\langle g,\pi(k,l)h\rangle\pi(k,l)\]
	defines a projection in $\mathcal{A}^1_1$. Hence $S_0(\RR)$ is an $\cA^1_1$-module and this projection represents $S_0(\RR)$ in the $K_0(\cA^1_1)$, the $K_0$ group 
	of the algebra of absolutely convergent Fourier series. We know that $K_0(\cA^1_1)$ is isomorphic to $\ZZ^2$ and has two generators. One generator corresponds to the trivial
element $1$ and the other to the Bott element. The projective module defined by $S_0(\RR)$ above corresponds to the Bott element, as can be seen by a computation
of the Connes-Chern character \cite{ri88}. However, since the $C^*$-completion of $\cA^1_1$ is isomorphic to $C(\TT^2)$ we know that there cannot be any non-trivial projections 
in $C(\TT^2)$ since $\TT^2$ is connected. Thus we have a contradiction, as the projection $p$ above would be non-trivial.
\end{proof}
This argument extends to higher dimensional symplectic lattices \cite{grhaheku02}, as well in much the
same way, since time-frequency shifts from symplectic lattices generate commutative algebras, which leads naturally to topological
obstructions and the results in \cite{ro08-3} might be useful in this context.
\\~\\
Note how the argument fails when we have a multi-window Gabor frame for $\ZZ^2$ generated by $g_1,...,g_n$ in $S_0(\RR)$. In that case, we can define a projection $P=(p_{ij})_{i,j=1}^n$ in the matrix algebra $M_n(C(\TT^2))$ over $C(\TT^2)$:
\[p_{ij}=\sum_{k,l\in\ZZ}\langle g_i,\pi(k,l)h_j\rangle\pi(k,l),\]
where $h_j$ denotes the canonical dual Gabor atom to the jth Gabor atom $g_j$. An elementary computation shows that $P$ is a projection in $M_n(C(\TT^2))$, which does not yield a contradiction, because there are non-trivial projections in $M_n(C(\TT^2))$ for $n\ge 2$. 

\end{document}